\documentclass{amsart}

\usepackage{amsmath,amssymb,amsthm}
\usepackage{tikz}

\newtheorem{thm}{Theorem}[section]
\newtheorem{theorem}[thm]{Theorem}
\newtheorem{lemma}[thm]{Lemma}
\newtheorem{proposition}[thm]{Proposition}
\newtheorem{corollary}[thm]{Corollary}
\newtheorem*{maintheorem}{Main Theorem}
\theoremstyle{definition}
\newtheorem{definition}[thm]{Definition}
\newtheorem{remark}[thm]{Remark}

\newtheorem{conjecture}[thm]{Conjecture}

\newcommand{\ind}[1]{\mathbf{1}(#1)}
\newcommand{\Ann}{\operatorname{Ann}}
\def\cm#1{\mathrm{M}\left(#1\right)}
\newcommand{\rank}{\operatorname{rank}}
\newcommand{\Kk}{\mathbb{K}}
\newcommand{\Kir}{\operatorname{Kir}}
\newcommand{\Rr}{\mathcal{R}}
\newcommand{\Aa}{\mathcal{A}}

\title[Strong Lefschetz Property of the Kirchhoff algebra of $K_{1,1,n}$]{Strong Lefschetz property of the Kirchhoff algebra of the complete tripartite graph $K_{1,1,n}$}

\author{Yohji Akama}
\address{Mathematical Institute, Graduate School of Science, Tohoku University, Aramaki, Sendai, Miyagi, 980-8578, Japan}
\email{yoji.akama.e8@tohoku.ac.jp}

\date{\today}



\begin{document}

\begin{abstract}
  By Whitney's 2-isomorphism theorem, the cycle matroids of two graphs are isomorphic if and only if one graph can be transformed into the other by finitely many applications of two operations: separation or identification at a cut-vertex, and twisting. The complete tripartite graph whose vertex set is the disjoint union of two singletons and an $n$-element set is invariant, up to isomorphism, under these operations. In this paper we prove that the Macaulay inverse system (the Kirchhoff algebra) determined by the basis generating polynomial of the cycle matroid of such a graph has the strong Lefschetz property over a field of characteristic $0$. The proof reduces the non-degeneracy of the higher Hessians to the ranks of disjointness matrices of families of subsets (the Gottlieb--Wilson theorem) and to the representation theory of the symmetric group in relation to the Johnson scheme.
\end{abstract}

\maketitle

\section{Introduction}

The strong Lefschetz property of graded Artinian rings originates from the hard Lefschetz theorem for the cohomology rings of complex projective varieties.
A standard graded Artinian Gorenstein algebra $A=\bigoplus_{k=0}^{d}A_k$ over a field $\Kk$ of characteristic $0$ is said to have the \emph{strong Lefschetz property} (SLP) if there exists $\ell\in A_1$ such that
\[
  \times\ell^{\,d-2k}\colon A_k\longrightarrow A_{d-k}
\]
is an isomorphism for every $0\le k\le d/2$. Such an $\ell$ is called a \emph{Lefschetz element}.
A standard graded Artinian Gorenstein algebra can be written as the quotient $\Rr/\Ann(f)$ of the ring $\Rr=\Kk[\partial_1,\dots,\partial_N]$ of differential operators by the annihilator ideal $\Ann(f)=\{\alpha\in\Rr:\alpha f=0\}$ of a homogeneous polynomial $f$ of degree $d$ in $N$ variables (Macaulay's inverse system).
By Maeno--Watanabe \cite{MR2594646}, the SLP is characterized by the non-degeneracy of the higher Hessians of $f$.
Namely, for a basis $\mathcal{B}_k=\{\alpha_1,\dots,\alpha_{\dim A_k}\}$ of $A_k$, consider the matrix $H_{k,\mathcal{B}_k}(f)=(\alpha_i\alpha_j f)_{i,j}$ whose entries are homogeneous polynomials of degree $d-2k$.
Then, for $p=(p_1,\dots,p_N)\in\Kk^N$, the element $\ell=\sum_i p_i\partial_i$ is a Lefschetz element if and only if $\det H_{k,\mathcal{B}_k}(f)(p)\ne0$ for all $k\le d/2$.
Here $H_{k,\mathcal{B}_k}(f)(p)$ denotes the matrix obtained from $H_{k,\mathcal{B}_k}(f)$ by substituting $x_i=p_i$ into each entry.
The case $k=1$ is the ordinary Hessian.

Maeno and Numata \cite{MR2985388,MR3566530} considered the Artinian Gorenstein algebra $\Rr/\Ann(f_M)$ determined by the basis generating polynomial
\[
  f_M=\sum_{B:\ \text{a basis of }M}\ \prod_{e\in B}x_e
\]
of a matroid $M$.
This algebra is called the \emph{Kirchhoff algebra} of $M$, and we denote it by $\Kir(M)$ in this paper.
They conjectured the following.
\begin{conjecture}[Maeno--Numata]\label{conj:MN}
For every matroid $M$, the Kirchhoff algebra $\Kir(M)$ has the strong Lefschetz property over a field of characteristic $0$.
\end{conjecture}

The bases of the cycle matroid $\cm{G}$ of a graph $G$ are the maximal spanning forests of $G$, which are the spanning trees if $G$ is connected.
In the latter case, the polynomial $f_{\cm{G}}$ is strictly log-concave, from which the condition for $k=1$ in the definition of the SLP follows (Nagaoka--Yazawa \cite{MR4234203}).
Yazawa \cite{MR4348920} proved that the cycle matroids of the complete graphs and the complete bipartite graphs with at most five vertices satisfy the Maeno--Numata conjecture, and Nagaoka-Wachi \cite{NagaokaWachi} demonstrated SLP for Artinian Gorenstein algebras generated by relative invariants of prehomogeneous vector spaces of commutative parabolic type; since the Kirchhoff polynomial of a complete graph is the determinant of a generic symmetric matrix, this covers the cycle matroids of all complete graphs.

However, Takahashi \cite{MR4974303}, using a computer program based on a database of 2-connected simple graphs, found a $2$-connected simple graph $G_T$ such that $\cm{G_T}$ is a counterexample to the Maeno--Numata conjecture (he explicitly constructed a basis $\mathcal{B}_3$ and an eigenvector corresponding to the eigenvalue $0$ of the higher Hessian $H_{3,\mathcal{B}_3}(f_{\cm{G_T}})$, i.e., the case $k=3$).
\begin{figure}[tb]
    \centering
    \begin{tikzpicture}[scale=.5,every node/.style={circle, fill=black, inner sep=2pt, minimum size=6pt}]
        \node(5) at (0,2){};
        \node(2) at (4,0){};
        \node(8) at (0,0){};
        \node(1) at (-4,0){};
        \node(4) at (2,-2){};
        \node(7) at (4,-4){};
        \node(3) at (-2,-2){};
        \node(6) at (-4,-4){};

        \draw[line width=1pt](1)--(5);
        \draw[line width=1pt](1)--(6);
        \draw[line width=1pt](1)--(7);
        \draw[line width=1pt](1)--(8);
        \draw[line width=1pt](2)--(5);
        \draw[line width=1pt](2)--(6);
        \draw[line width=1pt](2)--(7);
        \draw[line width=1pt](2)--(8);
        \draw[line width=1pt](3)--(6);
        \draw[line width=1pt](3)--(8);
        \draw[line width=1pt](4)--(7);
        \draw[line width=1pt](4)--(8);
        \draw[line width=1pt](5)--(8);
    \end{tikzpicture}
    \caption{The graph $G_T$.
    \label{figure:SLP_3 in introduction}}
\end{figure}
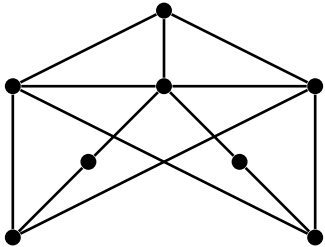
We therefore address the following problem:
\begin{quote}
Find an interesting infinite family $\{G_n\}_{n=1,2,3,\ldots}$ of graphs whose cycle matroids satisfy the Maeno--Numata conjecture.
\end{quote}

The graph $G_T$ of the counterexample $\cm{G_T}$ is not chordal (that is, it contains an induced cycle other than a triangle).
Moreover, $G_T$ contains the complete graph $K_4$ as a minor.
Indeed, in the drawing of $G_T$ in Figure~\ref{figure:SLP_3 in introduction}, a $K_4$ appears as a subgraph in the upper part of the graph obtained by contracting the leftmost vertical edge.
We therefore look for graphs $G$ satisfying the Maeno--Numata conjecture among the chordal graphs that do not contain $K_4$ as a minor.

Now, by Whitney's 2-isomorphism theorem (\cite[Section~5.3]{MR2849819}, \cite{MR1506961}), the cycle matroids of two graphs without isolated vertices are isomorphic if and only if one graph can be transformed into the other by finitely many applications of the following two operations:
\begin{enumerate}
\item[(i)] \emph{Identification of vertices and its inverse, separation at a cut-vertex}:
the operation of identifying two vertices belonging to different connected components, and its inverse, the operation of separating a graph into two graphs at a cut-vertex.
\item[(ii)] \emph{Whitney's twist}:
  Suppose that a graph $G$ is obtained from disjoint graphs $G_1$ and $G_2$ by identifying a vertex $u_1$ of $G_1$ with a vertex $u_2$ of $G_2$ into a vertex $u$ of $G$, and identifying a vertex $v_1$ of $G_1$ with a vertex $v_2$ of $G_2$ into a vertex $v$ of $G$.
  The twist of $G$ with respect to $\{u, v\}$ (also called the Whitney twist) is the graph obtained by instead identifying $u_1$ with $v_2$ and $v_1$ with $u_2$.
\end{enumerate}
For example, in Figure~\ref{fig:twist}, the graph $G'$ is a twist of $G$.
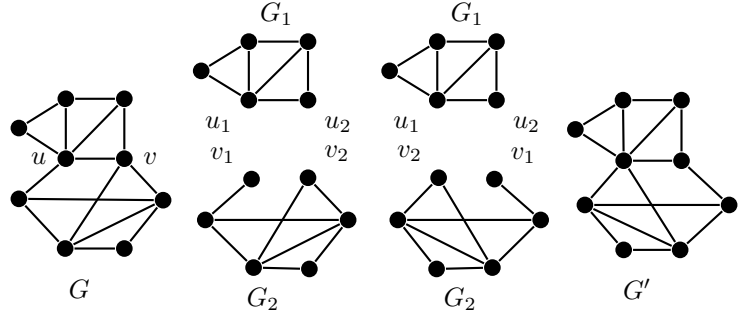
\begin{figure}[ht]\centering
\begin{tikzpicture}[
    vertex/.style={circle, fill=black, inner sep=0pt, minimum size=2.2mm},
    thick,
    scale=0.9
]

\node[vertex] (a) at (0.31,-1.79) {};
\node[vertex] (b) at (1.00,-1.36) {};
\node[vertex] (c) at (1.86,-1.36) {};
\node[vertex,label=left:$u$]  (u) at (1.00,-2.24) {};
\node[vertex,label=right:$v$] (v) at (1.86,-2.24) {};
\node[vertex] (d) at (0.31,-2.83) {};
\node[vertex] (g) at (2.43,-2.85) {};
\node[vertex] (e) at (0.99,-3.56) {};
\node[vertex] (f) at (1.86,-3.56) {};
\draw (a)--(b)--(c)--(v)--(u)--(a);   
\draw (b)--(u);
\draw (u)--(c);
\draw (u)--(d)--(e)--(f)--(g)--(v); 
\draw (d)--(g);
\draw (e)--(g);
\draw (e)--(v);
\node at (1.2,-4.15) {$G$};

\node[vertex] (A1) at (2.99,-0.95) {};
\node[vertex] (A2) at (3.69,-0.52) {};
\node[vertex] (A3) at (4.56,-0.52) {};
\node[vertex,label=below left:$u_1$]  (U1) at (3.69,-1.38) {};
\node[vertex,label=below right:$u_2$] (U2) at (4.56,-1.38) {};
\draw (A1)--(A2)--(A3)--(U2)--(U1)--(A1);
\draw (A2)--(U1);
\draw (U1)--(A3);
\node at (4.1,-0.12) {$G_1$};

\node[vertex,label=above left:$v_1$]  (V1) at (3.73,-2.54) {};
\node[vertex,label=above right:$v_2$] (V2) at (4.56,-2.52) {};
\node[vertex] (P) at (3.05,-3.14) {};
\node[vertex] (S) at (5.15,-3.14) {};
\node[vertex] (Q) at (3.76,-3.84) {};
\node[vertex] (R) at (4.58,-3.86) {};
\draw (V1)--(P)--(Q)--(R)--(S)--(V2);
\draw (P)--(S);
\draw (Q)--(S);
\draw (Q)--(V2);
\node at (3.9,-4.3) {$G_2$};

\node[vertex] (B1) at (5.76,-0.95) {};
\node[vertex] (B2) at (6.46,-0.52) {};
\node[vertex] (B3) at (7.33,-0.52) {};
\node[vertex,label=below left:$u_1$]  (W1) at (6.46,-1.38) {};
\node[vertex,label=below right:$u_2$] (W2) at (7.33,-1.38) {};
\draw (B1)--(B2)--(B3)--(W2)--(W1)--(B1);
\draw (B2)--(W1);
\draw (W1)--(B3);
\node at (6.9,-0.12) {$G_1$};

\node[vertex,label=above left:$v_2$]  (X2) at (6.48,-2.52) {};
\node[vertex,label=above right:$v_1$] (X1) at (7.30,-2.54) {};
\node[vertex] (Pm) at (7.98,-3.14) {};
\node[vertex] (Sm) at (5.88,-3.14) {};
\node[vertex] (Qm) at (7.27,-3.84) {};
\node[vertex] (Rm) at (6.45,-3.86) {};
\draw (X1)--(Pm)--(Qm)--(Rm)--(Sm)--(X2);
\draw (Pm)--(Sm);
\draw (Qm)--(Sm);
\draw (Qm)--(X2);
\node at (6.8,-4.3) {$G_2$};

\node[vertex] (a') at (8.49,-1.81) {};
\node[vertex] (b') at (9.19,-1.38) {};
\node[vertex] (c') at (10.05,-1.38) {};
\node[vertex] (u') at (9.20,-2.27) {};
\node[vertex] (v') at (10.05,-2.27) {};
\node[vertex] (s') at (8.63,-2.92) {};
\node[vertex] (p') at (10.75,-2.92) {};
\node[vertex] (q') at (10.03,-3.58) {};
\node[vertex] (r') at (9.20,-3.58) {};
\draw (a')--(b')--(c')--(v')--(u')--(a');
\draw (b')--(u');
\draw (u')--(c');
\draw (u')--(s')--(r')--(q')--(p')--(v');
\draw (s')--(p');
\draw (s')--(q');
\draw (u')--(q');
\node at (9.4,-4.15) {$G'$};
\end{tikzpicture}
\caption{A twist with respect to the pair $\{u,v\}$ (\cite[Fig.~5.9]{MR2849819}). \label{fig:twist}}
\end{figure}
Moreover, if we apply the operations of the 2-isomorphism theorem to the counterexample $G_T$ to the Maeno--Numata conjecture, then the resulting graph is isomorphic to $G_T$ itself.
The same holds for the complete tripartite graph $K_{1,1,n}$ ($n\ge1$), whose vertex set is the disjoint union of two singletons and one set of $n$ elements (Figure~\ref{fig:k11n}).
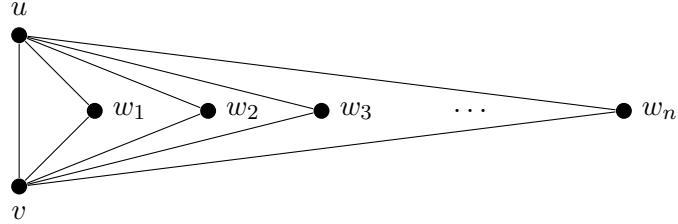
\begin{figure}[ht]
  \centering
  \begin{tikzpicture}[font=\large, every node/.style={circle, fill=black, inner sep=2pt, minimum size=6pt}]

  \node (u) [label=above:$u$] at (0, 2) {};
  \node (v) [label=below:$v$] at (0, 0) {};

  \draw (u) -- (v);

  \node (w1) [label=right:$w_1$] at (1, 1) {};
  \node (w2) [label=right:$w_2$] at (2.5, 1) {};
  \node (w3) [label=right:$w_3$] at (4, 1) {};
  \node (wn) [label=right:$w_n$] at (8, 1) {};

  \node [fill=none, draw=none, label=center:{$\dots$}] at (6, 1) {};

  \draw (u) to  (w1);
  \draw (v) to  (w1);
  \draw (u) to  (w2);
  \draw (u) to  (w3);
  \draw (v) to  (w2);
  \draw (v) to  (w3);
  \draw (u) to  (wn);
  \draw (v) to  (wn);

  \end{tikzpicture}
  \caption{The complete tripartite graph $K_{1,1,n}$. \label{fig:k11n}}
\end{figure}
The graph $K_{1,1,n}$ is chordal and does not contain $K_4$ as a minor.\footnote{A graph $G$ has no $K_4$ minor if and only if its treewidth is at most $2$, and this is in turn equivalent to $G$ being a series-parallel graph; these are well-known facts (see, for example, Duffin \cite{MR175809}, Bodlaender \cite{MR1647486}, and Oxley \cite{MR2849819}). We can also consider a series-parallel matroid $M$. A 2-connected
graph is series-parallel if and only if its cycle matroid is a series-parallel matroid, and a loopless connected matroid $M$ is
series-parallel if and only if Crapo’s beta invariant is 1 (see \cite{MR2849819}).}
In this paper we prove the following:

\begin{maintheorem}
    The cycle matroid $\cm{K_{1,1,n}}$ satisfies the Maeno--Numata conjecture for every $n\ge1$.
    That is, $\Kir(\cm{K_{1,1,n}})$ has the strong Lefschetz property over every field of characteristic $0$.
\end{maintheorem}

The outline of the proof is as follows.
In Section~\ref{sec:prelim} we recall the criterion for SLP in terms of higher Hessians, and in Section~\ref{sec:poly} we compute the basis generating polynomial $B_n$ of $K_{1,1,n}$ explicitly.
In Section~\ref{sec:gw} we discuss the Gottlieb--Wilson theorem (\cite{MR204305}, \cite{MR1078717}) on the rank of disjointness matrices. Moreover, for the permutation representation of the symmetric group $\mathfrak S_n$ on the subsets of $[n]=\{1,2,\ldots,n\}$, we consider its decomposition into irreducibles, in relation to the Johnson scheme.
In Section~\ref{sec:proof}, by a suitable change of coordinates, we write down a basis of $\Aa_k$ explicitly.
Then the value of the higher Hessian at a point $p$ is expressed as a block matrix built from the disjointness matrices of subsets of $[n]$.
The non-singularity of this block matrix is then proved by means of the results of Section~\ref{sec:gw}.

\section{Preliminaries}\label{sec:prelim}

Let $\Kk$ be a field of characteristic $0$, and let $\Rr=\Kk[X_1,\dots,X_N]$ act on $\Kk[x_1,\dots,x_N]$ by $X_i=\partial/\partial x_i$.
For a nonzero homogeneous polynomial $f$ of degree $d$, put $\Aa=\Rr/\Ann(f)$.
Then $\Aa=\bigoplus_{k=0}^{d}\Aa_k$ is a Gorenstein ring with socle degree $d$, and the pairing $\Aa_k\times\Aa_{d-k}\to\Aa_d\cong\Kk$ induced by the product is perfect.
The ring $\Rr$ is graded by giving each $X_i$ degree $1$, and $\Rr_k$ denotes its homogeneous component of degree $k$ (the subspace of all homogeneous polynomials of degree $k$).
Moreover, $\Aa_k\cong\Rr_k f=\{\alpha f:\alpha\in\Rr_k\}$.
Indeed, the map $\Rr_k\to\Rr_kf$, $\alpha\mapsto\alpha f$, is a surjective linear map whose kernel is $\Ann(f)\cap\Rr_k$.
Since $f$ is homogeneous, $\Ann(f)$ is a homogeneous ideal (the polynomials $\alpha_jf$ have different degrees for different homogeneous components $\alpha_j$), and $\Aa_k=\Rr_k/(\Ann(f)\cap\Rr_k)$, so this map induces an isomorphism $\Aa_k\cong\Rr_kf$.

\begin{lemma}[Maeno--Watanabe \cite{MR2594646}]\label{lem:MW}
Let $p=(p_1,\dots,p_N)\in\Kk^N$ and put $\ell=\sum_ip_iX_i\in\Aa_1$.
Let $(\alpha\beta f)(p)$ denote the value obtained by substituting $x_i=p_i$ ($1\le i\le N$) into the polynomial $\alpha\beta f$.
Let $m=d-2k\ge0$.
Then $\times\ell^{m}\colon\Aa_k\to\Aa_{d-k}$ is an isomorphism if and only if the bilinear form
\[
  Q_k(\alpha,\beta)=(\alpha\beta f)(p)\qquad(\alpha,\beta\in\Aa_k)
\]
is non-degenerate.
\end{lemma}

\begin{proof}
Since the pairing $\Aa_k\times\Aa_{d-k}\to\Kk$ is perfect, $\times\ell^m$ is an isomorphism if and only if the bilinear form $(\alpha,\beta)\mapsto\alpha\beta\ell^m f$ on $\Aa_k$ is non-degenerate (here $\alpha\beta\ell^m f$ is a constant, since $\alpha\beta\ell^m$ has degree $d$).
Since $\alpha\beta f$ is a homogeneous polynomial of degree $m$, we have $\ell^m(\alpha\beta f)=m!\,(\alpha\beta f)(p)$, and $m!\ne0$ because the characteristic is $0$.
\end{proof}

Hence $\Aa$ has the SLP with Lefschetz element $\ell$ if and only if $Q_k$ is non-degenerate for every $0\le k\le d/2$.

\section{The basis generating polynomial of $K_{1,1,n}$}\label{sec:poly}

Let $u,v,w_1,\dots,w_n$ be the vertices of $K_{1,1,n}$, and assign a variable $c$ to the edge $uv$, a variable $a_i$ to the edge $uw_i$, and a variable $b_i$ to the edge $vw_i$.
Every spanning tree has $n+1$ edges.
We distinguish cases according to whether the edge $uv$ is used.

\begin{itemize}
\item If $uv$ is used, then $u$ and $v$ are already joined, so each $w_i$ is attached by exactly one of $a_i$ and $b_i$. The contribution is $c\prod_{i=1}^n(a_i+b_i)$.
\item Otherwise, $u$ and $v$ are joined through exactly one $w_j$ using both $a_j$ and $b_j$, and each of the other $w_i$ is attached by exactly one of $a_i$ and $b_i$. The contribution is $\sum_{j=1}^n a_jb_j\prod_{1\le i\le n, i\ne j}(a_i+b_i)$.
\end{itemize}

Therefore
\begin{equation*}
  B_n=c\prod_{i=1}^n(a_i+b_i)+\sum_{j=1}^na_jb_j\prod_{1\le i\le n, i\ne j}(a_i+b_i),
\end{equation*}
which is a homogeneous polynomial of degree $n+1$.
Hereafter we put $\Aa=\Rr/\Ann(B_n)$; it is a Gorenstein ring with socle degree $d=n+1$.

\section{Disjointness matrices and the Gottlieb--Wilson theorem}\label{sec:gw}

For a condition $\Theta$, let $\ind{\Theta}$ denote $1$ if $\Theta$ holds and $0$ otherwise.

\begin{definition}[Disjointness matrix]\label{def:disj}
Let $X$ be a finite set and let $t,k$ be non-negative integers at most $|X|$.
We define the matrices
\[
  D_{t,k}(X)=\bigl(\ind{I\cap J=\emptyset}\bigr)_{I\in\binom{X}{t},\,J\in\binom{X}{k}},\qquad
  D_k(X)=D_{k,k}(X).
\]
In particular, for $X=[n]$ we write $D_{t,k}(n)=D_{t,k}([n])$ and $D_k(n)=D_k([n])$.
\end{definition}

The matrix $D_k(n)$ is the adjacency matrix of the Kneser graph $K(n,k)$ (see, e.g., \cite[p.~24]{MR3497070}).
The \emph{inclusion matrix} is
$$W_{t,k}(n)=\bigl(\ind{I\subseteq J}\bigr)_{I\in\binom{[n]}{t},\,J\in\binom{[n]}{k}}$$
(see \cite[p.~61, p.~115]{MR3497070}).

For a matrix $Y$, let $F_Y$ denote the linear map given by multiplying a column vector by $Y$ from the left, and let $\rank_{\Kk}Y$ denote the rank over $\Kk$ (for integer matrices it coincides with the rank over $\mathbb Q$, since $\operatorname{char}\Kk=0$).

\begin{theorem}[Wilson~\protect{\cite{MR1078717}}]\label{thm:GW}
Let $0\le t\le k\le n-t$.
  There exist unimodular integer matrices $P,P'$ such that $P\,W_{t,k}(n)\,P'$ is a diagonal matrix (with the remaining columns zero) whose diagonal entries are
  \[
    \binom{k-i}{t-i}\quad\Bigl(\text{with multiplicity }\tbinom ni-\tbinom n{i-1}\Bigr),\qquad i=0,1,\dots,t.
  \]
\end{theorem}

Since $\binom{k-i}{t-i}\ne0$ for $0\le i\le t\le k$, and the sum of the multiplicities is $\sum_{i=0}^{t}\bigl(\binom ni-\binom n{i-1}\bigr)=\binom nt$, we obtain the following.

\begin{corollary}\label{cor:fullrank}
Let $t\le k\le n-t$. Then $\rank_{\Kk}W_{t,k}(n)=\binom nt$
(that is, $W_{t,k}(n)$ has full row rank, the row vectors are linearly independent, and $F_{W_{t,k}(n)}$ is surjective).
\end{corollary}

Since $I\cap J=\emptyset$ is equivalent to $I\subseteq[n]\setminus J$, letting $P$ be the permutation matrix representing the bijection $\binom{[n]}{k}\to\binom{[n]}{n-k}$, $J\mapsto[n]\setminus J$, we have
\[
  D_{t,k}(n)=W_{t,n-k}(n)\,P.
\]

\begin{corollary}\label{cor:rankD}
For every non-negative integer $t\le\min(k,n-k)$ we have $\rank_{\Kk}D_{t,k}(n)=\binom nt$.
In particular, if $2k\le n+1$, then $D_{k-1,k}(n)$ has full row rank and $F_{D_{k-1,k}(n)}$ is surjective.
\end{corollary}

\begin{proof}
The condition $t\le\min(k,n-k)$ is equivalent to $t\le n-k\le n-t$, so it suffices to apply Corollary~\ref{cor:fullrank} to $W_{t,n-k}(n)$.
For $t=k-1$ the condition reads $k-1\le n-k$, that is, $2k\le n+1$.
\end{proof}

Next we describe the decomposition with respect to the action of the symmetric group $\mathfrak S_n$.
The group $\mathfrak S_n$ acts on $[n]$ by permutations, and hence on $\binom{[n]}{k}$ by $\sigma\cdot I=\{\sigma(i):i\in I\}$.
Let $(e_I)_{I\in\binom{[n]}{k}}$ be the standard basis of $\mathbb Q^{\binom{[n]}{k}}$; the formula $\sigma e_I=e_{\sigma\cdot I}$ defines a representation of $\mathfrak S_n$ (the permutation representation).
The $(I,J)$ entry $\ind{I\cap J=\emptyset}$ of $D_{t,k}(n)$ does not change when $(I,J)$ is replaced by $(\sigma\cdot I,\sigma\cdot J)$, so the linear map $F_{D_{t,k}(n)}\colon\mathbb Q^{\binom{[n]}{k}}\to\mathbb Q^{\binom{[n]}{t}}$ is $\mathfrak S_n$-equivariant.
The same holds for the inclusion matrix $W_{t,k}(n)$.
The irreducible decomposition of the permutation representation $\mathbb Q^{\binom{[n]}{k}}$ is given as follows by the theory of the Johnson scheme.

\begin{proposition}\label{prop:johnson}
Let $0\le k\le n$.
\begin{enumerate}
\item[(a)] $\mathbb Q^{\binom{[n]}{k}}=\bigoplus_{i=0}^{\min(k,n-k)}V_i^{(k)}$ is a multiplicity-free decomposition into mutually orthogonal irreducible representations, with
$V_i^{(k)}\cong V_i^{(i)}$ and $\dim V_i^{(k)}=\binom ni-\binom n{i-1}$, where $\binom n{-1}=0$.
\item[(b)] $F_{D_k(n)}$ acts on $V_i^{(k)}$ as multiplication by the scalar $(-1)^i\binom{n-k-i}{k-i}$.
\item[(c)] $F_{D_{t,k}(n)}$ maps $V_i^{(k)}$ to $V_i^{(t)}$ and maps all other components to $0$,
where $V_i^{(t)}$ is understood to be $0$ if $i>\min(t,n-t)$.
\end{enumerate}
\end{proposition}

\begin{proof}
  The following proof depends on nothing other than Corollary~\ref{cor:fullrank} (a corollary of Wilson's theorem).
  The statement follows from Young's rule \cite{Sagan}, but we give a direct proof for the reader's convenience.
Put $m=\min(k,n-k)$.

\emph{Step 0: Construction of the decomposition.}
For $0\le i\le m$ put $U_i^{(k)}=\operatorname{Im}F_{W_{i,k}(n)^{\mathsf T}}\subseteq\mathbb Q^{\binom{[n]}{k}}$.
Since $F_{W_{i,k}(n)^{\mathsf T}}$ is $\mathfrak S_n$-equivariant, $U_i^{(k)}$ is $\mathfrak S_n$-invariant.
The condition $i\le m$ is equivalent to $i\le k\le n-i$, so by Corollary~\ref{cor:fullrank} the matrix $W_{i,k}(n)$ has full row rank,
$F_{W_{i,k}(n)^{\mathsf T}}$ is injective, and $\dim U_i^{(k)}=\binom ni$.
For $1\le i\le k$ we have $W_{i-1,i}(n)W_{i,k}(n)=(k-i+1)W_{i-1,k}(n)$ (the $(A,J)$ entry of the left-hand side is the number $k-i+1$ of $i$-element sets $B$ with $A\subseteq B\subseteq J$ if $A\subseteq J$, and $0$ otherwise),
so, taking transposes, we obtain $U_{i-1}^{(k)}\subseteq U_i^{(k)}$.
Since $\dim U_m^{(k)}=\binom nm=\binom nk$, we have $U_m^{(k)}=\mathbb Q^{\binom{[n]}{k}}$.
Put $U_{-1}^{(k)}=0$ and, using orthogonal complements with respect to the standard inner product, define
\[
  V_i^{(k)}=U_i^{(k)}\cap\bigl(U_{i-1}^{(k)}\bigr)^{\perp}\qquad(0\le i\le m).
\]
Since the permutation representation preserves the standard inner product, $V_i^{(k)}$ is $\mathfrak S_n$-invariant,
$V_0^{(k)},\dots,V_m^{(k)}$ are mutually orthogonal, and
$\bigoplus_{i=0}^mV_i^{(k)}=\mathbb Q^{\binom{[n]}{k}}$, $\dim V_i^{(k)}=\binom ni-\binom n{i-1}$.
For $0\le i\le m\le n/2$ we have $\binom ni>\binom n{i-1}$ ($i\ge1$), so $V_i^{(k)}\ne0$.

\emph{Step 1: Each $V_i^{(k)}$ is irreducible, and they are pairwise non-isomorphic.}
The algebra $E=\operatorname{End}_{\mathfrak S_n}\bigl(\mathbb Q^{\binom{[n]}{k}}\bigr)$ consists of all matrices commuting with every permutation matrix.
A matrix $A$ commutes with all permutation matrices if and only if its entry $A_{I,J}$ is constant on the orbits of the action of $\mathfrak S_n$ on $\binom{[n]}{k}\times\binom{[n]}{k}$.
The orbit of $(I,J)$ is determined by $|I\cap J|$, and $|I\cap J|$ takes every value from $\max(0,2k-n)$ to $k$.
Hence $E$ has basis $A_j=\bigl(\ind{|I\cap J|=j}\bigr)_{I,J}$ ($\max(0,2k-n)\le j\le k$), and $\dim_{\mathbb Q}E=k-\max(0,2k-n)+1=m+1$.
Each $A_j$ is a symmetric matrix, so every element of $E$ is symmetric.
If $A,B\in E$, then $AB\in E$ is also symmetric, so $AB=(AB)^{\mathsf T}=B^{\mathsf T}A^{\mathsf T}=BA$, and $E$ is a commutative ring.
On the other hand, if the irreducible decomposition of $\mathbb Q^{\binom{[n]}{k}}$ is $\bigoplus_jN_j^{\,m_j}$ (the $N_j$ being pairwise non-isomorphic irreducible representations), then
$E\cong\prod_jM_{m_j}(D_j)$, where $D_j=\operatorname{End}_{\mathfrak S_n}(N_j)$ is a division algebra over $\mathbb Q$.
Since $E$ is commutative, all $m_j=1$.
Let $r=\sum_jm_j$ be the number of irreducible components.
Since $V_0^{(k)},\dots,V_m^{(k)}$ are all nonzero and give a direct sum of mutually orthogonal invariant subspaces, $r\ge m+1$, and
$m+1=\dim_{\mathbb Q}E=\sum_j\dim_{\mathbb Q}D_j\ge r\ge m+1$.
Hence $r=m+1$, and each $V_i^{(k)}$ consists of exactly one irreducible component, i.e., it is an irreducible representation.
Since the $N_j$ are pairwise non-isomorphic, the decomposition is multiplicity-free.
A multiplicity-free decomposition is uniquely determined as the decomposition into isotypic components.

\emph{Step 2: $V_i^{(k)}\cong V_i^{(i)}$.}
Let $0\le i\le m$. By Step 0 for $k=i$, $V_i^{(i)}=\bigl(\operatorname{Im}F_{W_{i-1,i}(n)^{\mathsf T}}\bigr)^{\perp}=\ker F_{W_{i-1,i}(n)}$ (for $i=0$, $V_0^{(0)}=\mathbb Q$).
Let $x\in V_i^{(i)}$. For $0\le l<i$ we have $W_{l,i-1}(n)W_{i-1,i}(n)=(i-l)W_{l,i}(n)$, hence $F_{W_{l,i}(n)}(x)=0$, that is,
\begin{equation}\label{eq:xker}
  \sum_{A\supseteq L}x_A=0\qquad(L\in\tbinom{[n]}{l},\ 0\le l<i),
\end{equation}
where the sum runs over $i$-element sets $A$.
We show that $F_{W_{i,k}(n)^{\mathsf T}}(x)\in U_i^{(k)}$ is orthogonal to $U_{i-1}^{(k)}$.
For $y\in\mathbb Q^{\binom{[n]}{i-1}}$, put $N(u)=\binom{n-u}{k-u}$ for $u\le k$ and $N(u)=0$ for $u>k$. Since the number of $k$-element sets $J$ containing $A\cup B$ is $N(|A\cup B|)$, we have
\[
  \bigl\langle W_{i,k}(n)^{\mathsf T}x,\,W_{i-1,k}(n)^{\mathsf T}y\bigr\rangle
  =\sum_{A,B}x_Ay_BN(|A\cup B|)
  =\sum_By_B\sum_Ax_A\,f(|A\cap B|),
\]
where $|B|=i-1$ and $f(u)=N(2i-1-u)$.
Since $f$ is a function on $\{0,\dots,i-1\}$, it can be written in the binomial basis as $f(u)=\sum_{l=0}^{i-1}c_l\binom ul$.
By \eqref{eq:xker}, $\sum_Ax_A\binom{|A\cap B|}{l}=\sum_{L\subseteq B,\,|L|=l}\sum_{A\supseteq L}x_A=0$ for $l\le i-1$.
Hence the inner product is $0$, and $F_{W_{i,k}(n)^{\mathsf T}}(V_i^{(i)})\subseteq V_i^{(k)}$.
The map $F_{W_{i,k}(n)^{\mathsf T}}$ is injective and $\mathfrak S_n$-equivariant, and $V_i^{(i)}\ne0$.
Since $V_i^{(k)}$ is irreducible (Step 1), the image is all of $V_i^{(k)}$, and we obtain $V_i^{(k)}\cong V_i^{(i)}$.
This proves (a).

\emph{Step 3: (b).}
Let $x\in V_i^{(i)}$. For a $k$-element set $J$ and an $i$-element set $A$, the number of $k$-element sets $J'$ with $J'\cap J=\emptyset$ and $J'\supseteq A$ is $\binom{n-k-i}{k-i}$ if $A\cap J=\emptyset$, and $0$ otherwise.
Hence $D_k(n)W_{i,k}(n)^{\mathsf T}=\binom{n-k-i}{k-i}D_{i,k}(n)^{\mathsf T}$.
Moreover, by inclusion--exclusion, $\ind{A\cap J=\emptyset}=\sum_{L\subseteq A\cap J}(-1)^{|L|}$, so using \eqref{eq:xker},
\begin{align*}
  \bigl(D_{i,k}(n)^{\mathsf T}x\bigr)_J&=\sum_Ax_A\ind{A\cap J=\emptyset}
  =\sum_{L\subseteq J}(-1)^{|L|}\sum_{A\supseteq L}x_A\\
  &=(-1)^i\sum_{L\subseteq J,\,|L|=i}x_L
  =(-1)^i\bigl(W_{i,k}(n)^{\mathsf T}x\bigr)_J
\end{align*}
(here the terms with $|L|<i$ vanish by \eqref{eq:xker}, and for $|L|>i$ there is no $i$-element set $A$ containing $L$).
Therefore $D_k(n)\,W_{i,k}(n)^{\mathsf T}x=(-1)^i\binom{n-k-i}{k-i}W_{i,k}(n)^{\mathsf T}x$.
By Step 2, $W_{i,k}(n)^{\mathsf T}$ maps $V_i^{(i)}$ onto $V_i^{(k)}$, so (b) follows.

\emph{Step 4: (c).}
Since $F_{D_{t,k}(n)}$ is $\mathfrak S_n$-equivariant, by Schur's lemma the image of $V_i^{(k)}$ is contained in the sum of the irreducible subrepresentations of $\mathbb Q^{\binom{[n]}{t}}$ isomorphic to $V_i^{(k)}$.
For $0\le i<j\le n/2$, the representations $V_i^{(i)}$ and $V_j^{(j)}$ are not isomorphic.
Indeed, by Step 2, $V_i^{(i)}\cong V_i^{(j)}$, and $V_i^{(j)}$ and $V_j^{(j)}$ are distinct components of the multiplicity-free decomposition (Step 1) of $\mathbb Q^{\binom{[n]}{j}}$.
Therefore, together with $V_i^{(k)}\cong V_i^{(i)}$, the only irreducible subrepresentation of $\mathbb Q^{\binom{[n]}{t}}$ isomorphic to $V_i^{(k)}$ is $V_i^{(t)}$ (and there is none if $i>\min(t,n-t)$).
Hence $F_{D_{t,k}(n)}$ maps $V_i^{(k)}$ to $V_i^{(t)}$.
\end{proof}

\section{Proof of the Main Theorem}\label{sec:proof}

Hereafter let $n\ge1$ and fix $0\le k\le(n+1)/2$.
By Lemma~\ref{lem:MW}, it suffices to show that $Q_k$ is non-degenerate for some $p$.

\subsection{A change of coordinates and a basis of $\Aa_k$}\label{subsec:basis}

Consider the invertible linear change of coordinates from $(a_1,\dots,a_n,b_1,\dots,b_n,c)$ to $(s_1,\dots,s_n,t_1,\dots,t_n,z)$ given by
\[
  \begin{aligned}
  &s_j=a_j+b_j,\qquad t_j=a_j-b_j\quad(1\le j\le n),\\
  &z=c+\frac14\sum_{j=1}^ns_j=\frac14\sum_{j=1}^na_j+\frac14\sum_{j=1}^nb_j+c.
  \end{aligned}
\]
For $I\subseteq[n]$ write $s_I=\prod_{i\in I}s_i$.
Since $a_jb_j=\frac14(s_j^2-t_j^2)$, the basis generating polynomial of $K_{1,1,n}$ in the new coordinates is
\begin{equation*}
  B_n=zs_{[n]}-\frac14\sum_{j=1}^nt_j^{2}s_{[n]\setminus\{j\}}.
\end{equation*}
Since $\Aa$ is determined by $B_n$ independently of the choice of linear coordinates, we may work in the new coordinates (Lemma~\ref{lem:MW} holds in any linear coordinate system).
We write the partial differential operators with respect to the new coordinates as
\[
  S_i=\partial_{s_i},\qquad T_i=\partial_{t_i},\qquad Z=\partial_z.
\]
For $I\subseteq[n]$ we write $S^I=\prod_{i\in I}S_i$ and $\overline{I}=[n]\setminus I$.

For $k\le n$, the results of applying monomials of degree $k$ in the operators to $B_n$ are as follows.

\begin{center}
\begin{tabular}{lll}
Operator $u$ & Condition & $uB_n$ \\ \hline
$ZS^I$ & $I\in\binom{[n]}{k-1}$ & $s_{\overline{I}}$ \\
$T_{i'}S^{I'}$ & $I'\in\binom{[n]\setminus\{i'\}}{k-1}$ & $-\frac12\,t_{i'}\,s_{\overline{I'}\setminus\{i'\}}$ \\
$S^{I''}$ & $I''\in\binom{[n]}{k}$ & $zs_{\overline{I''}}-\frac14\sum_{j\in\overline{I''}}t_j^2s_{\overline{I''}\setminus\{j\}}$ \\
$T_j^2S^I$ & $|I|=k-2,\ j\in\overline{I}$ & $-\frac12\,s_{\overline{I}\setminus\{j\}}=-\frac12\,(ZS^{I\cup\{j\}})B_n$ \\
\end{tabular}
\end{center}

All other operators, for example $Z^2$, $ZT_j$, $T_jT_{j'}$ ($j\ne j'$, $j,j'\in [n]$), $S_i^2$, and $T_{i'}S^{I'}$ with $i'\in I'$, annihilate $B_n$.

\begin{definition}\label{def:B}
Let $\mathcal{B}_k$ be the sequence obtained by concatenating the following three sequences (the indices are arranged in some fixed order):
\begin{gather*}
  \mathcal{E}_k(n)=(ZS^I)_{I\in\binom{[n]}{k-1}},\qquad
  \mathcal{G}_k(n)=(S^{I''})_{I''\in\binom{[n]}{k}},\\
  \mathcal{F}_k(n)=(T_{i'}S^{I'})_{i'\in[n],\ I'\in\binom{[n]\setminus\{i'\}}{k-1}}.
\end{gather*}
\end{definition}

\begin{lemma}\label{lem:basis}
The sequence $\mathcal{B}_k$ is a basis of $\Aa_k$.
\end{lemma}

\begin{proof}
We have $\Aa_k\cong\Rr_kB_n$, and monomials other than those in the table above annihilate $B_n$.
The result of applying $T_j^2S^I$ to $B_n$ is a constant multiple of that of $ZS^{I\cup\{j\}}$, so $\Rr_kB_n$ is spanned by the images of $\mathcal{B}_k$.
Moreover, the first three kinds of operators in the table are linearly independent.
Indeed, suppose that a linear combination of the three kinds of results vanishes.
Only the results for $S^{I''}$ contain the variable $z$, and the monomials $zs_{\overline{I''}}$ ($I''\in\binom{[n]}{k}$) are pairwise distinct, so the coefficients of $\mathcal{G}_k(n)$ are $0$.
The remaining results are the monomials $s_{\overline{I}}$ and $t_{i'}s_{\overline{I'}\setminus\{i'\}}$; they are pairwise distinct (the pair $(i',I')$ is recovered from the monomial $t_{i'}s_{\overline{I'}\setminus\{i'\}}$), so all the remaining coefficients are $0$ as well.
\end{proof}

\subsection{Evaluation of the $k$-th Hessian at the point $p$}\label{subsec:gram}
\begin{definition}
  Let $p$ be the point given in the new coordinates by
  \[
    (s_1,\dots,s_n,t_1,\dots,t_n,z)=(1,\dots,1,0,\dots,0,1).
  \]
\end{definition}
For the corresponding $\ell=Z+\sum_iS_i$, the form $Q_k(u,v)=(uvB_n)(p)$, $(u,v\in\mathcal{B}_k)$, of Lemma~\ref{lem:MW} is given by the following table.

\begin{center}
\renewcommand{\arraystretch}{1.3}
\begin{tabular}{c|ccc}
 & $\mathcal{E}_k(n):\ v=ZS^J$ & $\mathcal{F}_k(n):\ v=T_{j'}S^{J'}$ & $\mathcal{G}_k(n):\ v=S^{J''}$ \\[0.5\baselineskip]
\hline
$\mathcal{E}_k(n):\ u=ZS^{I}$ & $0$ & $0$ & $\ind{I\cap J''=\emptyset}$ \\
$\mathcal{F}_k(n):\ u=T_{i'}S^{I'}$ & $0$ & $-\frac12\,\ind{i'=j'\ \mbox{and}\ I'\cap J'=\emptyset}$ & $0$ \\
$\mathcal{G}_k(n):\ u=S^{I''}$ & $\ind{I''\cap J=\emptyset}$ & $0$ & $\ind{I''\cap J''=\emptyset}$ \\
\end{tabular}
\end{center}
Here $I,J\in\binom{[n]}{k-1}$, $i',j'\in[n]$, $I'\in\binom{[n]\setminus\{i'\}}{k-1}$, $J'\in\binom{[n]\setminus\{j'\}}{k-1}$, and $I'',J''\in\binom{[n]}{k}$. 
The entries involving $Z^2$ or $ZT_j$ vanish because these operators annihilate $B_n$.
The entry for $(\mathcal{F}_k(n), \ \mathcal{G}_k(n))$  is $0$ because the polynomial $T_{i'}S^{I'}S^{J''}B_n$ is a multiple of $t_{i'}$, and $t_{i'}=0$ at the point $p$.
Consequently we obtain the following.
\begin{proposition}
\[
 H_{k,\mathcal{B}_k}(B_n)(p)=\begin{pmatrix}O&O&M_k(n)\\O&C_k(n)&O\\M_k(n)^{\mathsf T}&O&D_k(n)\end{pmatrix},
\]
where
\[
  M_k(n)=D_{k-1,k}(n),\qquad
  C_k(n)=-\frac12\bigoplus_{j=1}^nD_{k-1}\bigl([n]\setminus\{j\}\bigr),
\]
and $D_{k-1}([n]\setminus\{j\})$ is obtained by applying Definition~\ref{def:disj} to the $(n-1)$-element set $[n]\setminus\{j\}$.
\end{proposition}

\subsection{Non-degeneracy of the bilinear form $Q_k$}\label{subsec:nondeg}

\begin{proposition}\label{prop:nondeg}
If $1\le k\le(n+1)/2$, then the matrix $ H_{k,\mathcal{B}_k}(B_n)(p)$ is non-singular.
\end{proposition}

\begin{proof}
All matrices in this proof have rational entries, and we regard them as matrices over $\mathbb Q$.

(1) \emph{$C_k(n)$ is non-singular, i.e., $F_{C_k(n)}$ is an isomorphism.}
The matrix $D_{k-1}([n]\setminus\{j\})$ is the adjacency matrix of the Kneser graph $K(n-1,k-1)$, and by Proposition~\ref{prop:johnson}(b) (applied with $n-1$ and $k-1$ in place of $n$ and $k$) its eigenvalues are $(-1)^i\binom{n-k-i}{k-1-i}$ ($0\le i\le k-1$).
Since $2k\le n+1$, we have $n-k-i\ge k-1-i\ge0$, so none of these is $0$.

(2) \emph{$F_{M_k(n)}$ is surjective.} This follows from Corollary~\ref{cor:rankD}.

(3) \emph{$\ker F_{M_k(n)}=V^{(k)}_k$, and the action of $F_{D_k(n)}$.}
Let $\mathbb Q^{\binom{[n]}{k}}=\bigoplus_{i\le\min(k,n-k)}V_i^{(k)}$.
By Proposition~\ref{prop:johnson}(c), $F_{M_k(n)}$ maps each $V_i^{(k)}$ to $V_i^{(k-1)}$.
Since $F_{M_k(n)}$ is surjective and $\mathbb Q^{\binom{[n]}{k-1}}=\bigoplus_{i=0}^{k-1}V_i^{(k-1)}$,
for each $0\le i\le k-1$ the map $V_i^{(k)}\to V_i^{(k-1)}$ is a surjection between irreducible representations, hence an isomorphism.
Since $V_k^{(k-1)}=0$, $V_k^{(k)}$ is mapped to $0$ by $F_{M_k(n)}$.
Therefore $\ker F_{M_k(n)}=V_k^{(k)}$.
Here, when $2k=n+1$, we interpret $V_k^{(k)}=0$.
By Proposition~\ref{prop:johnson}(b), if $2k\le n$, then $F_{D_k(n)}$ acts on $V_k^{(k)}$ as multiplication by the scalar $(-1)^k\binom{n-2k}{0}=(-1)^k$.

(4) \emph{The kernel is $0$.} Put $\Gamma_k(n)=H_{k,\mathcal{B}_k}(B_n)(p)$.
Let $(\xi,\eta,\zeta)\in\ker F_{\Gamma_k(n)}$,
where $\xi,\eta,\zeta$ belong to the linear spaces spanned by $\mathcal{E}_k(n),\mathcal{F}_k(n),\mathcal{G}_k(n)$, respectively.
From the first block row of $\Gamma_k(n)$ we get $F_{M_k(n)}(\zeta)=0$, so $\zeta\in\ker F_{M_k(n)}=V_k^{(k)}$.
From the second block row we get $F_{C_k(n)}(\eta)=0$, so $\eta=0$ by (1).
From the third block row we get $F_{M_k(n)^{\mathsf T}}(\xi)+F_{D_k(n)}(\zeta)=0$.
Now $F_{M_k(n)^{\mathsf T}}(\xi)\in\operatorname{Im}F_{M_k(n)^{\mathsf T}}=(\ker F_{M_k(n)})^{\perp}$,
while by (3), $F_{D_k(n)}(\zeta)\in V_k^{(k)}=\ker F_{M_k(n)}$.
Thus these two vectors are orthogonal and sum to $0$, so both are $0$.
Since $F_{D_k(n)}$ is invertible on $V_k^{(k)}$, we get $\zeta=0$.
By (2), $F_{M_k(n)^{\mathsf T}}$ is injective, so $\xi=0$.
\end{proof}

\begin{proof}[Proof of the Main Theorem]
For $k=0$ we have $Q_0(1,1)=B_n(p)=1\ne0$.
For $1\le k\le(n+1)/2$ we use Proposition~\ref{prop:nondeg}.
Hence, by Lemma~\ref{lem:MW}, the map
$\times\ell^{\,n+1-2k}\colon\Aa_k\to\Aa_{n+1-k}$ is an isomorphism for all $0\le k\le(n+1)/2$, and $\Aa$ has the SLP.
(The point $p$ and the matrix $H_{k,\mathcal{B}_k}(B_n)(p)$ are rational, and a rational matrix that is non-singular over $\mathbb Q$ is non-singular over every field of characteristic $0$; hence this holds over every such field $\Kk$.)
\end{proof}


\begin{remark}
By a similar method as above,  we can also obtain that cycle matroid $\cm{K_{2,n}}$ and the uniform matroids $U_{r,n}$ satisfy the Maeno--Numata conjecture for every $1\le r\le n$. Here Yazawa's computation \cite{MR4348920} covers $K_{2,n}$ only for $n\le3$ (at most five vertices).
\end{remark}

\end{document}